\documentclass[11pt]{amsart}

\usepackage{mathptmx,amsmath, amssymb, amscd, paralist,tabularx,supertabular,amsmath, verbatim, amsthm, amssymb, mathrsfs, manfnt,latexsym,amscd,graphicx,hyperref,color}

\usepackage{fullpage}

\DeclareMathOperator{\area}{area}

\DeclareMathOperator{\covol}{covol}
\DeclareMathOperator{\coarea}{coarea}

\DeclareMathOperator{\GL}{GL}

\DeclareMathOperator{\M}{M}

\DeclareMathOperator{\N}{N}

\DeclareMathOperator{\PGL}{PGL}

\DeclareMathOperator{\PSL}{PSL}
\DeclareMathOperator{\Ram}{Ram}

\DeclareMathOperator{\SL}{SL}

\DeclareMathOperator{\TGA}{TGA}
\DeclareMathOperator{\tr}{tr}

\DeclareMathOperator{\vol}{vol}

\def\pp{\mathfrak{p}}

\newcommand{\bfC}{\mathbf C}

\newcommand{\bfH}{\mathbf H}

\newcommand{\bfQ}{\mathbf Q}
\newcommand{\bfR}{\mathbf R}

\newcommand{\frakp}{\mathfrak{p}}
\newcommand{\frakq}{\mathfrak{q}}

\newcommand{\frakP}{\mathfrak{P}}
\newcommand{\frakQ}{\mathfrak{Q}}

\numberwithin{equation}{section}

\theoremstyle{remark}

\newtheorem{rmk}[equation]{Remark}

\theoremstyle{plain}
\newtheorem{thm}[equation]{Theorem}
\newtheorem{theorem}[equation]{Theorem}

\newtheorem{proposition}[equation]{Proposition}

\newtheorem{lemma}[equation]{Lemma}
\newtheorem{cor}[equation]{Corollary}

\newcommand{\abs}[1]{\left\vert#1\right\vert}
\newcommand{\set}[1]{\left\{#1\right\}}

\title{Areas of totally geodesic surfaces of hyperbolic $3$--orbifolds}

\author{Benjamin Linowitz}
\address{Department of Mathematics\\Oberlin College\\Oberlin, OH 44074}
\email{benjamin.linowitz@oberlin.edu}

\author{D. B. McReynolds}
\address{Department of Mathematics\\Purdue University\\West Lafayette, IN 47907}
\email{dmcreyno@math.purdue.edu, mille965@math.purdue.edu}

\author{Nicholas Miller}

\begin{document}


\begin{abstract} 
The geodesic length spectrum of a complete, finite volume, hyperbolic $3$--orbifold $M$ is a fundamental invariant of the topology of $M$ via Mostow--Prasad Rigidity. Motivated by this, the second author and Reid defined a two-dimensional analogue of the geodesic length spectrum given by the multiset of isometry types of totally geodesic, immersed, finite-area surfaces of $M$ called the geometric genus spectrum. They showed that if $M$ is arithmetic and contains a totally geodesic surface, then the geometric genus spectrum of $M$ determines its commensurability class. In this paper we define a coarser invariant called the totally geodesic area set given by the set of areas of surfaces in the geometric genus spectrum. We prove a number of results quantifying the extent to which non-commensurable arithmetic hyperbolic $3$--orbifolds can have arbitrarily large overlaps in their totally geodesic area sets.
\end{abstract}


\maketitle 


\section{Introduction}

Given a complete, hyperbolic $3$--orbifold $M$ with finite volume, the \emph{geodesic length spectrum} of $M$ is the multiset of lengths of closed geodesics on $M$. The length spectrum of $M$ determines the Laplace eigenvalue spectrum of $M$ (see \cite[Thm 2]{Kelmer-M1}) and thus determines (Laplace) spectral invariants like dimension, volume, and total scalar curvature. When $M$ is arithmetic, Chinburg--Hamilton--Long--Reid \cite{CHLR} proved that the length spectrum determines the commensurability class of $M$. It is an open question whether this holds in the setting of non-arithmetic hyperbolic $3$--manifolds. Motivated by this question, Futer--Millichap \cite{FM} constructed, for all sufficiently large $V$, pairs of non-commensurable non-arithmetic hyperbolic $3$--manifolds with volume approximately $V$ and whose length spectra agree up to length $V$.

In \cite{MR-GGS}, a two-dimensional analogue of the geodesic length spectrum of $M$ was introduced. The {\it geometric genus spectrum} is the set of isometry types of immersed, totally geodesic surfaces in $M$ considered with multiplicity. That every surface occurs with finite multiplicity follows from Thurston's work on pleated surfaces \cite[Cor 8.8.6]{Thurston-Notes}. Unlike the geodesic length spectrum, which is always infinite, the geometric genus spectrum can be finite or even empty. Indeed, most hyperbolic $3$--manifolds do not contain any immersed totally geodesic surfaces. It is here that arithmeticity becomes relevant, as an arithmetic hyperbolic $3$--manifold that contains a single totally geodesic surface must contain infinitely many pairwise non-commensurable totally geodesic surfaces (see \cite[Ch 9]{MR}). In \cite{MR-GGS}, the second author and Reid considered the class of arithmetic hyperbolic $3$--manifolds and showed that if $M_1,M_2$ are two such manifolds with equal (nonempty) geometric genus spectra then $M_1,M_2$ are commensurable.

In this paper, we will consider the coarser invariant of all areas of immersed, totally geodesic surfaces of $M$. We will call this invariant the {\it totally geodesic area set} and denote it by $\TGA(M)$. Via an inequality of Uhlenbeck \cite[Lemma 6]{Hass-Uhl}, $\TGA(M)$ is a discrete subset of the positive real numbers. In the case that $M$ is arithmetic and $\TGA(M)$ is nonempty, we do not in general expect that $\TGA(M)$ will determine the commensurability class of $M$. In \S \ref{section:construction}, for instance, we construct non-commensurable arithmetic Kleinian groups whose commensurators contain maximal, arithmetic Fuchsian groups having exactly the same coareas. One may therefore ask whether a two-dimensional analogue of \cite{FM} holds. That is, can the totally geodesic area sets of non-commensurable arithmetic hyperbolic $3$--orbifolds have arbitrarily large overlap? In this paper we answer this question in the affirmative by proving a  number of results which quantify the extent to which non-commensurable arithmetic hyperbolic $3$--orbifolds may have arbitrarily large overlaps in their totally geodesic area sets.

Throughout this paper we will make use of standard asymptotic notation and will use the Vinogradov symbol $f\ll g$ to indicate that there exists a positive constant $C$ such that $\abs{f}<C\abs{g}$.

\begin{theorem}\label{theorem:areaset}
Let $M$ be an arithmetic hyperbolic $3$--orbifold and $A_1<A_2<\cdots < A_s$ be positive real numbers contained in $\TGA(M)$. Let $F(V)$ denote the number of commensurability classes of hyperbolic $3$--orbifolds containing a representative $M'$ with $\{A_1,\dots,A_s\}\subset \TGA(M')$ and $\vol(M')<V$. Then for all sufficiently large $V$, we have $F(V)\gg \frac{V^{2/3}}{\vol(M)^{16}}$, where the implicit constant depends only on the set $\{A_1,\dots,A_s\}$.
\end{theorem}

\noindent As an immediate application of Theorem \ref{theorem:areaset} we obtain the following corollary.

\begin{cor}
Let  $A_1<A_2<\cdots < A_s$ be positive real numbers. If there exists an arithmetic hyperbolic $3$--orbifold which contains immersed totally geodesic surfaces of areas $A_1,\dots, A_s$ then in fact there exist infinitely many pairwise non-commensurable arithmetic hyperbolic $3$--orbifolds with this property.
\end{cor}

We are also able to obtain an upper bound for the maximum cardinality of the set of pairwise non-commensurable arithmetic hyperbolic $3$--orbifolds having bounded volume and totally geodesic area sets containing $\{A_1,\dots,A_s\}$. In fact, we will prove a stronger result and give an upper bound for the number of maximal arithmetic hyperbolic $3$--orbifolds containing immersed totally geodesic surfaces with areas $A_1,\dots, A_s$. Recall that an arithmetic hyperbolic $3$--orbifold is called \emph{maximal} if $\pi_1(M)$ is a maximal, arithmetic Kleinian group in the sense of Maclachlan--Reid \cite{MR}. Our result has the added benefit of making the dependence upon $A_1,\dots, A_s$ explicit.

\begin{theorem}\label{theorem:areasetupperbound}
Let $A_1<A_2<\cdots < A_s$ be positive real numbers and $G(V)$ denote the number of isometry classes of maximal arithmetic hyperbolic $3$--orbifolds $M$ with $\{A_1,\dots,A_s\}\subset \TGA(M)$ and $\vol(M)<V$. Then for any $\epsilon>0$ and all sufficiently large $V$, we have $G(V)< e^{c\log(A_1)^{1+\epsilon}}V^{26}$, where $c$ is a positive constant which depends only on $\epsilon$.
\end{theorem}

Theorems \ref{theorem:areaset} and \ref{theorem:areasetupperbound} both concern the behavior of totally geodesic area sets across commensurability classes and the set of maximal arithmetic hyperbolic $3$--orbifolds in a commensurability class. In \S \ref{section:countingcovers} we take a different perspective and fix an arithmetic hyperbolic $3$--orbifold $M$ whose totally geodesic area set contains a fixed set of real numbers $\{A_1,\dots, A_s\}$. The main result of \S \ref{section:countingcovers} is a lower bound for the number of covers of $M$ which have bounded volume and whose totally geodesic area sets also contain $\{A_1,\dots, A_s\}$.

\begin{thm}\label{CC-Sec:T2}
Let $M$ be an arithmetic hyperbolic $3$--orbifold and $A_1<A_2<\cdots < A_s$ be positive real numbers contained in $\TGA(M)$. Let $H(V)$ denote the number of covers $M'$ of $M$ such that $\{A_1,\dots,A_s\}\subset\TGA(M')$ and $\vol(M')\le V$. Then for all sufficiently large $V$ $$H(V)\gg \frac{V^{1/6}}{\vol(M)^{1/6}\left[\log(V)-\log(\vol(M))\right]^{1/2}},$$ where the implicit constant depends only on $K$, the field of definition of $M$.
\end{thm}


\paragraph{\textbf{Acknowledgements.}}

The authors would like to thank Paul Pollack for useful conversations on the material of this paper. The second author was partially supported by NSF grant DMS-1408458. The third author was partially supported by a Bilsland dissertation fellowship.

\section{Background}

\subsection{Notation}

Throughout this article $k$ will denote a number field and $K/k$ will be a relative quadratic extension. The degree of $k$ will be denoted $n$ and the ring of integers of $k$ will be denoted $\mathcal O_k$. Furthermore, we will denote by $d_k, h_k, \zeta_k(s)$ the absolute discriminant, class number, and Dedekind zeta function of $k$. The signature of $k$ will be denoted $(r_1,r_2)$ where $r_1$ is the number of real places of $k$ and $r_2$ is the number of complex places of $k$. Given an ideal $\mathscr I\subset  \mathcal O_k$ we will use $\N(\mathscr I)$, without any subscripts, to denote the norm of $\mathscr I$ down to $\bfQ$. We will always use subscripts to denote the norm of a relative extension of number fields (e.g., $\N_{K/k}(\mathscr J)$). Throughout this article $\log(x)$ will denote the natural logarithm function. Given a quaternion algebra $B$ over $k$, we will denote by $\Ram(B)$ the set of all places of $k$ at which $B$ is ramified. We will denote the set of finite (resp.~infinite) places of $k$ at which $B$ is ramified by $\Ram_f(B)$ (resp.~$\Ram_\infty(B)$). The discriminant $\mathscr D_B$ of $B$ is defined to be the product of all finite primes ramifying in $B$. We will denote by $k_B$ the class field associated to $B$. More explicitly, $k_B$ is the maximal abelian extension of $k$ which has $2$--elementary Galois group, is unramified outside of $\Ram_\infty(B)$, and in which every finite prime of $k$ which ramifies in $B$ splits completely. On the geometric side, we will denote by $\bfH^2$ and $\bfH^3$ real hyperbolic $2$-- and $3$--space. We will use $M$ to denote an arithmetic hyperbolic $3$--orbifold and $N$ to denote an arithmetic hyperbolic $2$--orbifold. That is, $M=\bfH^3/\Gamma_M$ and $N=\bfH^2/\Gamma_N$ where $\Gamma_M$ and $\Gamma_N$ are arithmetic lattices in $\PSL(2,\bfR)$ and $\PSL(2,\bfC)$. We will refer to lattices in $\PSL(2,\bfR)$ and $\PSL(2,\bfC)$ as being Fuchsian and Kleinian respectively.

\subsection{Number theoretic preliminaries}\label{section:ntprelims}

Let $k$ be a number field of degree $n$ with ring of integers $\mathcal O_k$. Given a square-free ideal $ \mathscr I\subset \mathcal O_k$, define two functions

 \[\Phi_1( \mathscr I)=\prod_{\mathfrak p \mid \mathscr I} \frac{\N(\mathfrak p) - 1}{2},\quad \Phi_2( \mathscr I)=\prod_{\mathfrak p\mid \mathscr I}\left(\N(\mathfrak p) + 1\right).\]
We will see in \S \ref{section:volumeformula} that these functions arise in the formula for the covolume of maximal arithmetic subgroups of $\PSL(2,\mathbf R)^a\times\PSL(2,\mathbf C)^b$. In this section we will record two results about these functions which we will later use to analyze the covolume of certain arithmetic groups. Both of these lemmas will make use of the following lemma of Belolipetsky--Gelander--Lubotzky--Shalev \cite[Lemma 3.4]{BGLS}.

\begin{lemma}[Belolipetsky--Gelander--Lubotzky--Shalev]\label{lemma:BGLS}
Let $I_k(X)$ denote the number of ideals of $\mathcal O_k$ of norm less than $X$, then $I_k(X)<\zeta_k(2)X^2$.
\end{lemma}

\begin{lemma}\label{lemma:bound1}
The number of square-free ideals $\mathscr I \subset \mathcal O_k$ with $\Phi_1(\mathscr I)\le X$ is at most $10^4\zeta_k(2) X^4$.
\end{lemma}
\begin{proof}
We will show that $\N(\mathscr I)<100\Phi_1(\mathscr I)^2$. The result will then follow from Lemma \ref{lemma:BGLS}. Observe that \[\frac{\Phi_1(\mathscr I)^2}{\N(\mathscr I)}=\prod_{\mathfrak p \mid \mathscr I} \N(\mathfrak p)^{-1}\left(\frac{\N(\mathfrak p) - 1}{2}\right)^2=\prod_{\mathfrak p \mid \mathscr I} \frac{\N(\mathfrak p)-2+\N(\mathfrak p)^{-1}}{4}.\] The terms in the latter product are strictly greater than $1$ whenever $\N(\mathfrak p)\geq 7$, and an easy calculation of the values for $\N(\mathfrak p) \in \{2,3,4,5\}$ shows that $\frac{\Phi_1(\mathscr I)^2}{\N(\mathscr I)}>\frac{1}{100}$, which is what we wanted to show.
\end{proof}

\begin{lemma}\label{lemma:bound2}
The number of square-free ideals $\mathscr I \subset \mathcal O_k$ with $\Phi_2(\mathscr I)\le X$ is at most $\zeta_k(2)X^2$.
\end{lemma}
\begin{proof}
This is an immediate consequence of Lemma \ref{lemma:BGLS} and the trivial bound $\Phi_2(\mathscr I)>\N(\mathscr I)$. 
\end{proof}

\subsection{Arithmetic lattices}

In this subsection we give a brief review of the construction of arithmetic lattices acting on $(\bfH^2)^a \times (\bfH^3)^b$ and refer the reader interested in a more detailed discussion to the text of Maclachlan--Reid \cite{MR}. To begin, we fix a number field $k$ of signature $(r_1, r_2)$ and a $k$--quaternion algebra $B$  which is not totally definite (i.e. $V_k^\infty \not\subset \mathrm{Ram}(B)$ where $V_k^\infty$ is the set of archimedean place of $k$). Under these assumptions, we have an isomorphism $B\otimes_{\bfQ} \bfR \cong \M(2,\bfR)^a \times \mathbb H^r \times \M(2,\bfC)^b$, where $r=\abs{\Ram_\infty(B)}$, $a=r_1-r$, and $b=r_2$. This isomorphism induces an injective homomorphism
\[ B^\times \hookrightarrow \prod_{\nu\notin \Ram_\infty(B)} (B\otimes_k k_\nu)^\times \longrightarrow \GL(2,\bfR)^a\times \GL(2,\bfC)^b.\] 
Restricting to the elements $B^1$ of $B^\times$ with reduced norm $1$ gives us an injective homomorphism 
\[\pi\colon B^1\hookrightarrow \SL(2,\bfR)^a\times \SL(2,\bfC)^b.\]
Given a maximal order $\mathcal O$ of $B$, by work of Borel--Harish-Chandra \cite{BHC}, $P(\pi(\mathcal O^1))$ is a lattice in $\PSL(2,\bfR)^a\times \PSL(2,\bfC)^b$. Finally, we say that an irreducible lattice $\Gamma  \subset \PSL(2,\bfR)^a\times \PSL(2,\bfC)^b$ is {\it arithmetic} if $\Gamma$ is commensurable with a lattice of the form $P(\pi(\mathcal O^1))$. When $a+b=1$, $\Gamma$ is an arithmetic Fuchsian group or an arithmetic Kleinian group.  For a discrete, finitely generated subgroup $\Gamma$ of either $\PSL(2,\mathbf{R})$ or $\PSL(2,\mathbf{C})$, the {\it trace field} of $\Gamma$ is the field given by $\bfQ(\tr \gamma : \gamma\in\Gamma)$. Although the trace field of $\Gamma$ is not an invariant of the commensurability class, it turns out that the trace field of the subgroup $\Gamma^2=\{\gamma^2 : \gamma\in\Gamma\}$ is a commensurability class invariant. We denote the trace field of $\Gamma^2$ by $k\Gamma$ and call it the {\it invariant trace field} of $\Gamma$. We may also define an algebra over the invariant trace field $k\Gamma$ by $B\Gamma := \left\{ \sum b_i \gamma_i : b_i\in k\Gamma, \gamma_i\in\Gamma^2\right\}$, where each sum is finite. Multiplication in $B\Gamma$ is defined in the obvious manner: $(b_1\gamma_1)\cdot (b_2\gamma_2) := (b_1b_2)(\gamma_1\gamma_2)$. The algebra $B\Gamma$ is a quaternion algebra which is also an invariant of the commensurability class of $\Gamma$. We call $B\Gamma$ the {\it invariant quaternion algebra} of $\Gamma$. Note that the invariant trace field and invariant quaternion algebras are complete commensurability class invariants in the sense that if $\Gamma_1$ and $\Gamma_2$ are arithmetic lattices then they are commensurable if and only if $k\Gamma_1\cong k\Gamma_2$ and $B\Gamma_1\cong B\Gamma_2$.

Note that the invariant trace field of $P(\pi(\mathcal O^1))$ is $k$ and the invariant quaternion algebra of $P(\pi(\mathcal O^1))$ is $B$. It follows that if $\Gamma$ is commensurable with $P(\pi(\mathcal O^1))$ then its invariant trace field and invariant quaternion algebra are also $k$ and $B$.


\subsection{Maximal arithmetic lattices}\label{subsection:maxarith}

We now briefly describe the construction of maximal arithmetic lattices in the commensurability class given by the arithmetic data $(k,B)$. This construction is given in more detail in Borel \cite{Borel}, Chinburg--Friedman \cite[p.~41]{CF-S}, and Maclachlan--Reid \cite[Ch 11]{MR}.

Let $S$ be a finite set of primes of $k$ which is disjoint from $\Ram_f(B)$. Given a prime $\mathfrak p\in S$, fix an edge $\{M_{\mathfrak p}^1,M_{\mathfrak p}^2\}$ in the tree of maximal orders of $B\otimes_k k_{\mathfrak p}\cong \M(2,k_{\mathfrak p})$ (i.e., in the affine building associated to $\SL(2,k_\mathfrak p)$, which in this case has the structure of a tree). More algebraically, let $\pi_\mathfrak p$ be a uniformizer for $k_\mathfrak p$. Then we are fixing two maximal orders $\{M_{\mathfrak p}^1,M_{\mathfrak p}^2\}$ with the property that as $\mathcal O_{k_\mathfrak p}$--modules, $M_{\mathfrak p}^1/M_{\mathfrak p}^1\cap M_{\mathfrak p}^2 = \mathcal O_{k_\mathfrak p}/\pi_\mathfrak p\mathcal O_{k_\mathfrak p}$. 

Let $\mathcal O$ be a maximal order of $B$. Given a prime $\mathfrak p$ of $k$, denote by $\mathcal O_\mathfrak p$ the maximal order $\mathcal O\otimes_{\mathcal O_k} \mathcal O_{k_\mathfrak p}$ of $B\otimes_k k_\mathfrak p$. Define a subgroup $\Gamma_{S,\mathcal O}\subset \PSL(2,\bfR)^a\times \PSL(2,\bfC)^b$ by intersecting the preimage in $\PGL(2,\bfR)^a\times \PGL(2,\bfC)^b$ of \[ \{\overline{x}\in B^\times/k^\times : x\mathcal O_\mathfrak p x^{-1}=\mathcal O_\mathfrak p\text{ for $\mathfrak p\not\in S$ and $x$ fixes $\{M_{\mathfrak p}^1,M_{\mathfrak p}^2\}$ for $\mathfrak p\in S$}\},\] with $\PSL(2,\bfR)^a\times \PSL(2,\bfC)^b$. 

It is a theorem of Borel \cite{Borel} that every maximal arithmetic subgroup of $\PSL(2,\bfR)^a\times \PSL(2,\bfC)^b$ in the commensurability class defined by $(k,B)$ is of the form $\Gamma_{S,\mathcal O}$. We note however, that the converse is false. Not every group of the form $\Gamma_{S,\mathcal O}$ is maximal. In the case that $S=\emptyset$ it is clear that $\Gamma_{S,\mathcal O}$ simply corresponds to the normalizer of $\mathcal O$. We will denote this group by $\Gamma_\mathcal O$. We call this group a {\it minimal covolume group} because, as will be seen in \S \ref{section:volumeformula}, $\Gamma_\mathcal O$ has minimal covolume amongst all arithmetic lattices in the commensurability class given by $(k,B)$.

\subsection{The volume formula}\label{section:volumeformula}

In this section we give a formula of Borel \cite{Borel} for the covolume of maximal arithmetic lattices arising from quaternion algebras and use the formula to prove two analytic results which will be needed in the proofs of Theorem \ref{theorem:areaset} and Theorem \ref{theorem:areasetupperbound}. It was shown in \cite{Borel} (see also \cite[Prop. 2.1]{CF}) that if $\Gamma_{S,\mathcal O}$ is a maximal arithmetic subgroup of $\PSL(2,\bfR)^a\times\PSL(2,\bfC)^b$ then
\begin{equation}\label{equation:maximumarithmeticarea}
\covol(\Gamma_{S,\mathcal O})=\frac{2(4\pi)^a d_k^{3/2}\zeta_k(2)}{(4\pi^2)^{r}(8\pi^2)^{b}}\cdot \frac{\Phi_1(\mathscr D_B)\Phi_2(\mathscr D_S)}{2^m[k_B:k]},
\end{equation}
where $\Phi_1, \Phi_2$ are as defined in Section \ref{section:ntprelims}, $r=\abs{\Ram_\infty(B)}$, $0\leq m \leq \abs{S}$, and $\mathscr D_S=\prod_{\mathfrak p\in S} \mathfrak p$. Note that the integer $m$ can be explicitly determined (c.f. \cite[pp.~355--356]{MR}).

\begin{lemma}\label{lemma:rambound}
Suppose that $\Gamma$ is an arithmetic Fuchsian group of coarea $X$ arising from a quaternion algebra $B/k$. Then there is a positive constant $c$, depending only on $k$, such that $\abs{\Ram_f(B)}<c\log(X)$.
\end{lemma}

\begin{proof}
Let $\mathcal O$ be a maximal order of $B$. Because $\Gamma_\mathcal{O}$ has minimal coarea amongst the arithmetic Fuchsian groups commensurable with $\Gamma$, we obtain from (\ref{equation:maximumarithmeticarea}) that
\[X=\coarea(\Gamma)\geq \coarea(\Gamma_\mathcal{O}) = \frac{8\pi d_k^{3/2}\zeta_k(2)}{(4\pi^2)^{n-1} [k_B:k]}\cdot \Phi_1(\mathscr D_B),\] where $n=[k:\mathbf Q]$. Note that by definition, $k_B$ is contained in the narrow class field of $k$, which has degree $2^nh_k$ over $k$ and therefore $[k_B:k]\leq 2^nh_k$. Hence we deduce from the inequality above that there exists a constant $C$, depending only on $k$, such that $\Phi_1(\mathscr D_B) < CX$. Let $r=\abs{\Ram_f(\mathscr D_B)}$, then 
\[\frac{\N(\mathscr D_B)}{4^r}=\prod_{\mathfrak p \mid \mathscr D_B} \frac{\N(\mathfrak p)}{4} \leq \prod_{\mathfrak p \mid \mathscr D_B}\frac{\N(\mathfrak p) -1}{2} =\Phi_1(\mathscr D_B) < CX.\]
We now have two cases to consider. Suppose first that $\N(\mathscr D_B)<5^r$, then \[\N(\mathscr D_B)=\prod_{\substack{\mathfrak p \mid \mathscr D_B\\ \N(\mathfrak p)\leq 9}}\N(\mathfrak p)\cdot \prod_{\substack{\mathfrak p \mid \mathscr D_B\\ \N(\mathfrak p)\geq 11}}\N(\mathfrak p) < 5^r.\] We note that the first product is trivially bounded below by $1$, whereas the second product is bounded below by $10^{r-7n}$. Indeed for the bound on the second product, we note that each term in the product is greater than $10$ and the total number of terms is $r-x$ where $x=\abs{\{\mathfrak p \in\Ram_f(B) : \N(\mathfrak p) \leq 9\}}$. Because there are at most $n$ primes of $k$ with the same (fixed) norm, we have $x\leq 7n$. The bound for the second product follows. This shows that $10^{r-7n}<5^r$. Straightforward manipulations now show that $r<24n$, giving us an upper bound depending only on the degree of $k$.

We now consider the second case: $\N(\mathscr D_B)\geq 5^r$. Combining this inequality with our previous inequality $\frac{\N(\mathscr D_B)}{4^r}<CX$, we easily obtain $r<c\log(X)$ where $c$ is a positive constant. The lemma now follows from this and the previous case.
\end{proof}

\begin{lemma}\label{lemma:maxvolbound}
Suppose that $\Gamma_{S,\mathcal O}$ is a maximal arithmetic Kleinian group arising from a quaternion algebra $A/K$. Then there is a positive constant $c$ such that $\Phi_1(\mathcal D_A)\Phi_2(\mathcal D_S)<c\cdot\covol(\Gamma_{S,\mathcal O})^{24}$. In fact, one may take $c=2^{45}3^{360}$.
\end{lemma}
\begin{proof} 
Setting $n=[K:\bfQ]$, by \eqref{equation:maximumarithmeticarea}, we have
\begin{equation}\label{equation:eq1}
\covol(\Gamma_{S,\mathcal O})=\frac{d_K^{3/2}\zeta_K(2)}{(4\pi^2)^{n-1}}\cdot \frac{\Phi_1(\mathscr D_A)\Phi_2(\mathscr D_S)}{2^m[K_A:K]},
\end{equation}
 for some positive integer $m\leq \abs{S}$. As $K_A$ is unramified at all finite primes it is contained in the narrow class field of $K$, whose degree (over $K$) is bounded above by $h_K\cdot 2^{n-2}$. In \cite[Lemma 3.1]{L-BLMS} it was shown that $h_K\leq 242\cdot d_K^{3/4}/(1.64)^{n-2}$. Combining this inequality with the estimates $2^m\leq 2^{\abs{S}}$ and $\zeta_K(2)\geq 1$, we obtain from \eqref{equation:eq1} that 
\begin{equation}\label{equation:eq2}
\frac{\Phi_1(\mathscr D_A)\Phi_2(\mathscr D_S)}{2^{\abs{S}}} \leq \frac{242\cdot 49^{n-1}\cdot \covol(\Gamma_{S,\mathcal O})}{d_K^{3/4}}.
\end{equation}
Although one has the trivial estimate $d_K\geq 1$, which could be applied to \eqref{equation:eq2}, one can obtain a much stronger bound by employing the Odlyzko bounds \cite{O} (see also \cite[\S 2]{BD}), which in our context imply that $d_K\geq e^{4n-6.5}$. Substituting this bound into \eqref{equation:eq2} and simplifying now gives us
\begin{equation}\label{equation:eq3}
\frac{\Phi_1(\mathscr D_A)\Phi_2(\mathscr D_S)}{2^{\abs{S}}} \leq 2^{15}\cdot 3^{n}\cdot \covol(\Gamma_{S,\mathcal O}).
\end{equation}
Note that \[\frac{\Phi_2(\mathscr D_S)}{2^{\abs{S}}}=\prod_{\mathfrak p\mid \mathscr D_S} \frac{\N(\mathfrak p)+1}{2},\] hence $\Phi_2(\mathscr D_S)^{1/3}\leq \Phi_2(\mathscr D_S)/2^{\abs{S}}$. Furthermore, we have the bound $\Phi_1(\mathscr D_A)\leq 2^{3n}\Phi_1(\mathscr D_A)^3$. Combining these with \eqref{equation:eq3} yields

\begin{equation}\label{equation:eq4}
\Phi_1(\mathscr D_A)\Phi_2(\mathscr D_S) \leq 2^{45}\cdot 3^{6n}\cdot \covol(\Gamma_{S,\mathcal O})^3.
\end{equation}
The proof now follows from \cite[Lemma 4.3]{CF}, which implies that $n< 60+3\log(\covol(\Gamma_{S,\mathcal O}))$.
\end{proof}

\subsection{Totally geodesic surfaces in arithmetic hyperbolic $3$--orbifolds}

Let $M$ be an arithmetic hyperbolic $3$--orbifold and $N$ be an immersed totally geodesic surface of $M$. Thus $N$ is an arithmetic hyperbolic surface for which $\pi_1(N) < \pi_1(M)$. In this brief section we will review some of the ways that the arithmetic invariants of $M$ and $N$ are related. We begin by stating a required result from \cite[Cor 9.5.3]{MR}.

\begin{proposition}\label{proposition:totallygeodesicinvariants}
Let $(A, K)$ be the invariant quaternion algebra and invariant trace field of $M$ and $(B,k)$ be the invariant quaternion algebra and invariant trace field of $N$. Then
\begin{enumerate}[(i)]
\item $[K:k]=2$ and $k=K\cap \bfR$,
\item $A\cong B \otimes_k K$.
\end{enumerate}
\end{proposition}

As an application of Proposition \ref{proposition:totallygeodesicinvariants} we show that every totally geodesic surface of $M$ has the same invariant trace field.

\begin{lemma}\label{lemma:uniquetracefield}
Let $M$ be an arithmetic hyperbolic $3$--orbifold with invariant trace field $K$ and $N$ an immersed totally geodesic surface. The invariant trace field $k$ of $N$ is the maximal totally real subfield of $K$. In particular, if $N, N'$ are totally geodesic surfaces of $M$ then the invariant trace fields of $N$ and $N'$ coincide.
\end{lemma}
\begin{proof}
The field $k$ is a totally real number field which, by Proposition \ref{proposition:totallygeodesicinvariants}, satisfies $[K:k]=2$. Let $F\subset K$ be a totally real subfield of $K$. The compositum $kF$ of $k$ and $F$ is a totally real subfield of $K$ which contains $k$, hence $kF=k$ or $kF=K$. As $K$ is not totally real, $kF=k$ and $F\subset k$. It follows that $k$ is the maximal totally real subfield of $K$.
\end{proof}

In light of Proposition \ref{proposition:totallygeodesicinvariants}(ii), it is of interest to determine when a quaternion algebra $A$ over $K$ is of the form $A\cong B\otimes_k K$. This is given by the following theorem \cite[Thm 9.5.5]{MR} (see \cite[Lemma 3.2]{LS} for a more general result valid over arbitrary number fields).

\begin{theorem}\label{theorem:tgs}
Let $K$ be a number field with a unique complex place and suppose that the maximal totally real subfield $k$ of $K$ satisfies $[K:k]=2$. Suppose $B$ is a quaternion algebra over $k$ ramified at all real places of $k$ except at the place lying under the complex place of $K$. Then $A \cong B \otimes_{k} K$ if and only if $\Ram_f(A)$ consists of $2r$ places $\set{\mathfrak{P}_{i,j}}_{1\le i\le r,\,1\le j\le 2}$ satisfying $\mathfrak{P}_{1,j}\cap \mathcal{O}_{k} = \mathfrak{P}_{2,i} \cap \mathcal{O}_{k} = \pp_i$, where $\set{\pp_1,\dots,\pp_r} \subset \Ram_f(B)$ with $\Ram_f(B)\setminus \set{\pp_1,\dots ,\pp_r}$ consisting of primes in $\mathcal{O}_{k}$ which are inert or ramified in $K/k$.
\end{theorem}

When the above conditions on $\Ram(A)$ are satisfied there will be infinitely many isomorphism classes of quaternion algebras $B$ over $k$ such that $A\cong B\otimes_k K$. In particular an arithmetic hyperbolic $3$--orbifold which contains a single immersed totally geodesic surface contains infinitely many primitive, totally geodesic, incommensurable surfaces.

\begin{cor}\label{cor:tensorup}
Let $K$ be a number field with a unique complex place and suppose that the maximal totally real subfield $k$ of $K$ satisfies $[K:k]=2$. Let $B_1,\dots, B_s$ be quaternion algebras over $k$ such that \[B_1\otimes_k K\cong B_2\otimes_k K\cong \cdots \cong B_s\otimes_k K.\] If $S$ is a finite set of primes of $k$, then the number of number fields $K'$ with $d_{K'}<x$ and which satsfy the following conditions:
\begin{enumerate}[(i)]
\item $K'$ has a unique complex place, and this place lies above the real place of $k$ which splits in all the $B_i$,
\item $[K':k]=2$,
\item every prime $\mathfrak p\in S$ decomposes the same way in $K'/k$ as it does in $K/k$,
\item $B_1\otimes_k K'\cong B_2\otimes_k K'\cong \cdots \cong B_s\otimes_k K'$,
\end{enumerate}
is greater than $cx$ as $x\to\infty$, where $c$ is a positive constant which depends only on $k$ and $t=\abs{S} + \sum_{i=1}^s \abs{\Ram_f(B_i)}$.
\end{cor}
\begin{proof}
We begin by noting that if we enlarge $S$ so that it contains $\bigcup_{i=1}^s\Ram_f(B_i)$, then by Theorem \ref{theorem:tgs}, any relative quadratic extension $K'$ of $k$ satisfying conditions (i)-(iii) must also satisfy condition (iv). The corollary now follows from \cite[Cor 3.14]{CDO} and the remark immediately following its proof.\end{proof}

\begin{proposition}\label{proposition:constructinggroups}
Let $\Gamma$ be a maximal arithmetic Kleinian group arising from a quaternion algebra $A/K$ and let $k$ be the maximal totally real subfield of $K$. Let $\Gamma_1,\dots, \Gamma_s$ be arithmetic Fuchsian groups contained in $\Gamma$ which arise from quaternion algebras $B_1/k,\dots, B_s/k$. If $K'$ is a number field satisfying conditions (i) - (iv) of Corollary \ref{cor:tensorup} then there exists a maximal arithmetic subgroup arising from $A'=B_1\otimes_k K'$ containing arithmetic Fuchsian subgroups with coareas $\coarea(\Gamma_1),\dots,\coarea(\Gamma_s)$.
\end{proposition}
\begin{proof}
For each $i\in\{1,\dots, s\}$, let $\mathcal O_i$ be a maximal order contained in $B_i$. Given a prime $\mathfrak p$ of $k$ not ramifying in $B_i$, fix an isomorphism $f_\mathfrak p^i\colon B_i\otimes_k k_{\mathfrak p} \rightarrow \M(2,k_{\mathfrak p})$ such that $f_\mathfrak p^i(\mathcal O_i)=\M(2,\mathcal O_{k_\mathfrak p})$. Let $S$ be a finite set of primes of $k$ containing

\begin{enumerate}
\item $\bigcup_{i=1}^s \Ram_f(B_i)$, and 
\item all primes $\mathfrak p\not\in \Ram_f(B_i)$ of $k$ for which the closure of $\Gamma_i$ in $P((B_i\otimes_k k_\mathfrak p)^*)$ does not have image in $\PGL(2,k_\mathfrak p)$ coinciding with $\PGL(2,\mathcal O_{k_\mathfrak p})$ for some $i\in\{1,\dots, s\}$.
\end{enumerate}
Let $K'$ be a quadratic extension of $k$ which has a unique complex place (which lies above the real place of $k$ splitting in all of the $B_i$) and in which every prime $\mathfrak p\in S$ has the same splitting behavior as it does in $K/k$. By Theorem \ref{theorem:tgs} there exist primes $\mathfrak p_1,\dots,\mathfrak p_r$ of $k$ such that the $2r$ primes of $K$ which ramify in $A$ are precisely the primes of $K$ lying above $\mathfrak p_1,\dots,\mathfrak p_r$. By construction of $K'$, the primes $\mathfrak p_1,\dots,\mathfrak p_r$ split in $K'/k$. Let $A'$ be the quaternion algebra over $K'$ which is ramified at all real places of $K'$ and the $2r$ primes lying above $\mathfrak p_1,\dots,\mathfrak p_r$. By Theorem \ref{theorem:tgs}, $B_i\otimes_k K'\cong A'$ for $1\leq i \leq s$.

Let $S_K$ (resp.~$S_{K'}$) denote the set of primes of $K$ (resp.~$K'$) lying above the primes of $S$. Because the primes contained in $S$ split the same way in the extensions $K/k$ and $K'/k$, there exists a bijection $\Phi\colon S_K \rightarrow S_{K'}$ such that $K_\mathfrak P\cong K'_{\Phi(\mathfrak P)}$ for all $\mathfrak P \in S_K$. As $\Phi(\Ram_f(A))=\Ram_f(A')$ and as over a non-archimedean local field there is a unique isomorphism class of quaternion division algebras, we may extend these field isomorphisms so as to obtain isomorphisms $A\otimes_K K_\mathfrak P\cong A'\otimes_{K'} K'_{\Phi(\mathfrak P)}$ for all $\mathfrak P \in S_K$.

We now define a maximal arithmetic subgroup $\Gamma'$ of $P(A'^*)$ by specifying that its closure in $P((A'\otimes_{K'} K'_{\mathfrak P'})^*)$ be $\PGL(2,\mathcal O_{k_{\mathfrak P'}})$ if $\mathfrak P'\not\in S_{K'}$ and that its closure corresponds to the image of $\Gamma$ under the identification $A\otimes_K K_\mathfrak P\cong A'\otimes_{K'} K'_{\Phi(\mathfrak P)}$ otherwise. Note that this local-global correspondence is valid for, and in fact uniquely characterizes, congruence arithmetic subgroups of quaternion algebras. All maximal arithmetic subgroups are congruence (see for instance \cite[Lemma 4.2]{LMR}).

By definition of $S$ and group $\Gamma'$, the group $\Gamma'$ is a maximal arithmetic Kleinian group arising from $A'$ which contains arithmetic subgroups isomorphic to $\Gamma_1,\dots,\Gamma_s$. As coareas may be computed locally using local Tamagawa volumes (see \cite[\S 6]{Borel} or \cite[Ch 11]{MR}), the proposition follows.
\end{proof}

\begin{rmk}\label{rmk:dependence}
We note that the set $S$ appearing in the proof of Proposition \ref{proposition:constructinggroups} can be taken so that its cardinality depends only on the coareas $\{\coarea(\Gamma_1),\dots,\coarea(\Gamma_s)\}$. Indeed, Lemma \ref{lemma:rambound} gives us a logarithmic bound for $\sum_{i=1}^s \abs{\Ram_f(B_i)}$. To give a bound on the second condition defining the set $S$, we note that the work of Borel \cite{Borel} shows that there are only finitely many arithmetic Fuchsian groups of bounded coarea. It follows that there exists a finite set of primes outside of which every arithmetic Fuchsian group with bounded coarea has local closure coinciding with $\PGL(2,\mathcal O_{k_\mathfrak p})$. \end{rmk}


\section{Proof of Theorem \ref{theorem:areaset}}

We now prove Theorem \ref{theorem:areaset}. Since $M$ is an arithmetic hyperbolic $3$--orbifold containing totally geodesic surfaces of areas $A_1,\dots,A_s$, $\pi_1(M)$ contains arithmetic Fuchsian subgroups $\Gamma_1,\dots, \Gamma_s$ such that the coarea of $\Gamma_i$ is $A_i$. Denote by $B_i$ the invariant quaternion algebra of $\Gamma_i$ and by $k_i$ the invariant trace field of $\Gamma_i$. Proposition \ref{proposition:totallygeodesicinvariants} and Lemma \ref{lemma:uniquetracefield} show that all of the $k_i$ are equal and in fact are the maximal totally real subfield of the invariant trace field $K$ of $M$. Denote by $k$ the maximal totally real subfield of $K$ and note that by Proposition \ref{proposition:totallygeodesicinvariants}, $[K:k]=2$. Let $K'$ be a number field which satisfies conditions (i)-(iv) in the statement of Corollary \ref{cor:tensorup}. 

We may assume without loss of generality that $\pi_1(M)=\Gamma_{S,\mathcal O}$ is a maximal arithmetic Kleinian group arising from a quaternion algebra $A$ over $K$. By Proposition \ref{proposition:constructinggroups}, the quaternion algebra $A'=B_1\otimes_k K'$ gives rise to a maximal arithmatic Kleinian group $\Gamma_{S',\mathcal O'}$ which  contains arithmetic Fuchsian subgroups with coareas $A_1,\dots, A_s$. Moreover, the proof of Proposition \ref{proposition:constructinggroups} shows that $\Phi_1(\mathscr D_{A'})=\Phi_1(\mathscr D_A)$ and $\Phi_2(\mathscr D_{S'})\leq \Phi_2(\mathscr D_S)$. Because of this, \eqref{equation:maximumarithmeticarea} and Lemma \ref{lemma:maxvolbound} show that there is a positive constant $c$ such that $\covol(\Gamma_{S',\mathcal O'})\leq c\cdot \vol(M)^{24}\cdot d_{K'}^{3/2}$. The lower bound now follows from Corollary \ref{cor:tensorup}, which shows that the number of choices of $K'$ with $d_{K'}<X$ is at least $cX$ as $X\to\infty$.


\section{Proof of Theorem \ref{theorem:areasetupperbound}}

We begin our proof of Theorem \ref{theorem:areasetupperbound} by noting that if there do not exist any arithmetic hyperbolic $3$--orbifolds with totally geodesic area sets containing $\{A_1,\dots,A_s\}$ then the statement of the theorem is trivially satisfied. Suppose therefore that $M$ is an arithmetic hyperbolic $3$--orbifold containing totally geodesic surfaces of areas $A_1,\dots,A_s$. Let $\Gamma_1,\dots, \Gamma_s$ be arithmetic Fuchsian groups contained in $\pi_1(M)$ whose coareas are $A_1,\dots, A_s$. As was pointed out in the proof of Theorem \ref{theorem:areaset}, all of the $\Gamma_i$ have the same invariant trace field, which we denote by $k$. Then $k$ is a totally real field whose degree we will denote by $n$. The proof of \cite[Thm 4.1]{L-BLMS} shows that the absolute discriminant $d_k$ of $k$ satisfies $d_k < A_1^{22}$. We now employ a theorem of Ellenberg--Venkatesh \cite{EV} in order to count the number of possibilities for the field $k$ (see also \cite[Appendix]{BEV}).

\begin{theorem}[Ellenberg--Venkatesh]\label{theorem:EV} 
Let $N(X)$ denote the number of isomorphism classes of number fields with absolute value of discriminant less than $X$. Then for any $\epsilon>0$ there is a constant $c(\epsilon)$ such that $\log N(X) \leq c(\epsilon)(\log X)^{1+\epsilon}$ for all $X\geq 2$.
\end{theorem}

By combining Theorem \ref{theorem:EV} with the bound $d_k\leq A_1^{22}$ we conclude that there are at most $e^{c\log(A_1)^{1+\epsilon}}$ possibilities for $k$, where the constant $c$ is allowed to depend on $\epsilon$. In other words, we have just shown that there are at most $e^{c\log(A_1)^{1+\epsilon}}$ number fields which may serve as the invariant trace field of arithmetic Fuchsian groups with coareas $A_1,\dots, A_s$. Fix one such field $k$. We now obtain an upper bound for the number of quadratic extensions of $k$ which may serve as the invariant trace field of an arithmetic Kleinian group with covolume at most $V$. 

Let $K$ be a quadratic extension of $k$ and suppose that $\mathscr C$ is a commensurability class of arithmetic Kleinian groups defined over $K$ such that the minimal covolume group $\Gamma_\mathcal O\in\mathscr C$ satisfies $\covol(\Gamma_\mathcal O)<V$. We now have
\begin{equation}\label{equation:ubeq1}
\covol(\Gamma_\mathcal O) = \frac{d_K^{3/2}\zeta_K(2)\Phi_1(\mathscr D_A)}{(4\pi^2)^{2n-1}[K_A:K]}<V,
\end{equation}
where $A$ is the invariant quaternion algebra of $\Gamma_\mathcal O$. By employing the trivial bounds $\zeta_k(2)\geq 1, \Phi_1(\mathscr D_A)\geq \frac{1}{2^{2n}}$ along with the bound $[K_A:K]\leq h_K 2^{2n-2} \leq 242(1.220)^{2n-2}d_K^{3/4}$ from the proof of Lemma \ref{lemma:maxvolbound} (in the paragraph following \eqref{equation:eq1}), we obtain
\begin{equation}\label{equation:ubeq2}
d_K < \left[3872\pi^2(32\pi)^{2n-2}\right]^{4/3}\cdot V^{4/3}.
\end{equation}
In order to estimate the number of quadratic extensions of $k$ which satisfy \eqref{equation:ubeq2} we will employ the following result of Cohen, Diaz y Diaz, and Olivier \cite[Cor 3.14]{CDO}.

\begin{theorem}[Cohen, Diaz y Diaz, and Olivier]\label{theorem:CDO}
Let $k$ be a number field of signature $(r_1,r_2)$ and $\mathfrak m_\infty$ be a set of real places of $k$. The number of quadratic extensions $K/k$ in which the real places of $k$ ramified in $K/k$ is equal to $\mathfrak m_\infty$ and such that $d_K<X$ is asymptotic to $\frac{d_k^2}{2^{r_1+r_2}}\cdot \frac{\kappa}{\zeta_k(2)}\cdot X$, where $\kappa$ is the residue at $s=1$ of the Dedekind zeta function $\zeta_k(s)$ of $k$.
\end{theorem}

\begin{rmk}
Two comments about Theorem \ref{theorem:CDO} are in order. The first is to point out that the asymptotic expression in the theorem turns out to be independent of the set $\mathfrak m_\infty$ of real places of $k$ which will ramify in the quadratic extension $K/k$. The second comment is that in \cite{CDO}, they prove their result by counting quadratic extensions $K/k$ satisfying the conditions in the theorem such that $\N_{k/\bfQ}(\Delta_{K/k})<X$, not such that $d_K<X$. Here $\Delta_{K/k}$ is the relative discriminant of the quadratic extension $K/k$. However since $d_K=\N_{k/\bfQ}(\Delta_{K/k})d_k^2$, our statement is equivalent.
\end{rmk}

Recall that $k$ is a totally real number field of degree $n$ and that we are interested in counting the number of quadratic extensions $K$ of $k$ which may serve as the invariant trace field of an arithmetic Kleinian group of covolume less than $V$. Because such a number field $K$ must have a unique complex place, its signature is $(2n-2,1)$. It now follows from Theorem \ref{theorem:CDO} that the number of extensions $K/k$ satisfying the bound in \eqref{equation:ubeq2} is at most $\kappa d_k^2 \cdot 12 \cdot 16^{4n} \cdot V^{4/3}$ for sufficiently large $V$, where $\kappa$ is the residue at $s=1$ of $\zeta_k(s)$. By \cite[p.~322]{L-ANT} there exists an absolute constant $c>0$ such that $\kappa\leq c^n d_k^{1/2}$. Combining this with our previous estimate we see that the number of extensions $K/k$ satisfying the bound in \eqref{equation:ubeq2} is at most $C^n\cdot d_k^{5/2} \cdot V^{4/3}$, where $C>0$ is an absolute constant. We have already seen that $d_k<A_1^{22}$ and $n<60+3\log(A_1)$ by \cite[Lemma 4.3]{CF}. Putting all of this together we see that there exists an absolute constant $C_1$ such that our bound for the number of extensions $K/k$ being considered is of the form $A_1^{C_1} V^{4/3}$.

Having given a bound on the number of quadratic extensions of $k$ which satisfy the bound in \eqref{equation:ubeq2}, fix one such extension $K$. We will now bound the number of quaternion algebras $A$ over $K$ which could give rise to a minimal covolume group $\Gamma_\mathcal O$ satisfying \eqref{equation:ubeq1}. Bounding $[K_A:K]$ as above, we obtain from \eqref{equation:ubeq1} that 
\begin{equation}\label{equation:ubeq3}
\Phi_1(\mathscr D_A) < 5 (5\pi^2)^{2n} V.
\end{equation}
As $\mathscr D_A$ is a square-free integral ideal of $\mathcal O_K$, it follows from Lemma \ref{lemma:bound1} and \eqref{equation:ubeq3} that there are at most $50^4(5\pi^2)^{8n}\zeta_K(2)V^4$ many choices for $\mathscr D_A$. It is well-known that the Dedekind zeta function satisfies the inequality $\zeta_K(s)<\zeta(s)^{[K:\bfQ]}$ for all $s>1$. Since $\zeta(2)=\pi^2/6$ we may simplify our upper bound on the number of choices for $\mathscr D_A$ to $C_2^n V^4$ where $C_2$ is an absolute constant. By \cite[Lemma 4.3]{CF} we have $n<60+3\log(A_1)$, therefore there exists an absolute constant $C_3$ such that our upper bound on the number of choices for $\mathscr D_A$ is of the form $A_1^{C_3} V^4$. Recall that any quaternion algebra over $K$ which gives rise to a Kleinian group must be ramified at all real places of $K$ \cite[Ch 8]{MR}. As isomorphism classes of quaternion algebras $A$ over $K$ are in one to one correspondence with the sets $\Ram(A)$ of places of $K$ ramifying in $A$, it follows that the isomorphism class of any quaternion algebra over $K$ which gives rise to a Kleinian group is given by $\Ram_f(A)$, the set of finite primes of $K$ which ramify in $A$. As $\mathscr D_A$ is simply the product of all primes lying in $\Ram_f(A)$, we see that there are at most $A_1^{C_3} V^4$ many isomorphism classes of quaternion algebras over $K$ which could give rise to a minimum covolume group satisfying \eqref{equation:ubeq1}. These isomorphism classes correspond, by \cite[Ch 8]{MR}, to the commensurability classes of arithmetic Kleinian groups with invariant trace field $K$ and invariant quaternion algebra $A$ which contain a representative with covolume less than $V$.

We may now regard $K$ and $A$ as being fixed. Then \cite[Thm 5.1]{L-BLMS} shows that the number of maximal arithmetic Kleinian groups arising from $A$ with covolume at most $V$ is less than $242 V^{20}$. We have now exhibited all of the bounds needed to prove Theorem \ref{theorem:areasetupperbound} and need only assemble them into our final bound.

We have shown that there are at most $e^{c\log(A_1)^{1+\epsilon}}$ possibilities for the totally real field $k$ over which our arithmetic Fuchsian groups will be defined, $A_1^{C_1} V^{4/3}$ possibilities for the quadratic extension $K/k$ over which the fundamental group of our maximal arithmetic hyperbolic $3$--orbifold will be defined, $A_1^{C_3} V^4$ possibilities for the quaternion algebra $A$, and $242 V^{20}$ possibilities for the fundamental group of our maximal arithmetic hyperbolic $3$--orbifold. By enlarging the constant $c$ as necessary, we therefore see that for any $\epsilon>0$ the number of isometry classes of maximal arithmetic hyperbolic $3$--orbifolds containing totally geodesic surfaces of areas $A_1,\dots, A_s$ is at most $e^{c\log(A_1)^{1+\epsilon}}V^{26}$, where $c$ is a constant depending only on $\epsilon$.


\section{Proof of Theorem \ref{theorem:areasetupperbound}}\label{section:construction}

Let $k$ be a totally real quartic field with narrow class number one and $p,q$ be distinct rational primes which both split completely in the extension $k/\bfQ$. Let $\frakp, \frakp'$ be primes of $k$ lying above $p$ and $\frakq, \frakq'$ be primes of $k$ lying above $q$. Let $K$ be a quadratic extension of $k$ which has signature $(6,1)$ and in which $\frakp,\frakp',\frakq,\frakq'$ all split. Let $\frakP_1,\frakP_2$ (respectively $\frakP_1',\frakP_2'$) be primes of $K$ lying above $\frakp$ (respectively $\frakp'$). Similarly, let $\frakQ_1,\frakQ_2$ (respectively $\frakQ_1',\frakQ_2'$) be primes of $K$ lying above $\frakq$ (respectively $\frakq'$). Let $A_1$ be the quaternion algebra over $K$ which is ramified at every real place of $K$ and satisfies $\Ram_f(A_1)=\{\frakP_1,\frakP_2,\frakQ_1,\frakQ_2\}$. Let $A_2$ be the quaternion algebra over $K$ which is ramified at every real place of $K$ and satisfies $\Ram_f(A_2)=\{\frakP_1',\frakP_2',\frakQ_1',\frakQ_2'\}$.

Let $S_1$ (respectively $S_2$) be the set of all quaternion algebras $B$ over $k$ such that $B\otimes_k K\cong A_1$ (respectively $B\otimes_k K\cong A_2$). We will now define a bijection $f\colon S_1\to S_2$. Suppose that $B\in S_1$; that is, $B\otimes_k K\cong A_1$. By Theorem \ref{theorem:tgs}, $B$ is ramified at all real places of $k$ not lying beneath the complex place of $K$. Furthermore, $\Ram_f(B)$ consists of $\frakp,\frakq$ and a set of primes of $k$ none of which split in $K/k$. In particular neither $\frakp'$ nor $\frakq'$ lies in $\Ram_f(B)$. Define $B'$ to be the quaternion algebra over $k$ whose ramification is exactly the same as that of $B$ though with $\frakp, \frakq$ replaced by $\frakp',\frakq'$. The classification of quaternion algebras over number fields shows that $B'$ is uniquely defined (up to isomorphism). Moreover, Theorem \ref{theorem:tgs} shows that $B\otimes_k K\cong A_2$ and therefore the map $f\colon S_1\to S_2$ is easily seen to be a bijection.

For each $B\in S_1$, we will now show that the maximal arithmetic Fuchsian groups arising from $B$ have the same areas as those arising from $f(B)$. Let $T_1$ be the set of maximal arithmetic Fuchsian groups arising from $B$ and $T_2$ be the set of maximal arithmetic Fuchsian groups arising from $f(B)$. We will define a bijection $g\colon T_1\to T_2$ which is area preserving. Using the notation developed in \S \ref{subsection:maxarith}, let $\Gamma_{S,\mathcal O}\in T_1$. Recall that this means that $\mathcal O$ is a maximal order of $B$ and $S$ is a set of finite primes of $k$ disjoint from $\Ram_f(B)$. Define $g(\Gamma_{S,\mathcal O})=\Gamma_{S',\mathcal O'}$ where $\mathcal O'$ is a maximal order of $f(B)$ and $S'$ consists of exactly the same primes as $S$, though with $\frakp$ replaced by $\frakp'$ and $\frakq$ replaced by $\frakq'$. We claim that $\area(\Gamma_{S,\mathcal O})=\area(\Gamma_{S',\mathcal O'})$. Indeed, this follows from \eqref{equation:maximumarithmeticarea} as $N(\frakp)=N(\frakp'), N(\frakq)=N(\frakq')$, and $[k_B:k]=[k_{f(B)}:k]=1$. The latter is due to the fact that $k_B$ and $k_{f(B)}$ are both extensions of $k$ which are contained in the narrow class field of $k$, which by hypothesis is simply $k$. Note that the number $m$ appearing in \eqref{equation:maximumarithmeticarea} will coincide for both groups because of the way that $m$ is defined and the fact that $k$ has trivial narrow class group (see \cite[pp.~355--356]{MR}). Finally, let $\mathcal O_1$ be a maximal order of $A_1$ and $\mathcal O_2$ be a maximal order of $A_2$. Let $\Gamma_1=P(\pi(\mathcal O_1^1))$ and $\Gamma_2=P(\pi(\mathcal O_2^1))$. The groups $\Gamma_1$ and $\Gamma_2$ are not commensurable, hence their commensurators $\mathrm{Comm}(\Gamma_1)=P(A_1^*)$ and $\mathrm{Comm}(\Gamma_2)=P(A_2^*)$ are distinct. On the other hand, the above discussion makes clear that $\mathrm{Comm}(\Gamma_1)$ contains maximal arithmetic Fuchsian groups with exactly the same areas as those of $\mathrm{Comm}(\Gamma_2)$.


\section{Proof of Theorem \ref{CC-Sec:T2}}\label{section:countingcovers}

In this section, we will count congruence subgroups of a fixed arithmetic Kleinian group $\Gamma$ that contains a set of Fuchsian subgroups with specified co-areas. To start, we fix an arithmetic lattice $\Gamma < \PSL(2,\mathbf{C})$ and a finite set $\mathcal{A} = \set{A_1,\dots,A_s}$ of real numbers such that for each $A_i \in \mathcal{A}$, there exists a Fuchsian subgroup $\Delta_i < \Gamma$ with $\coarea(\Delta_i) = A_i$. Setting $K$ to be the field of definition of $\Gamma$, for each ideal $\mathfrak{a} < \mathcal{O}_K$ we have a group homomorphism $R_\mathfrak{a}\colon \Gamma \to \PSL(2,\mathcal{O}_K/\mathfrak{a})$ given by reducing the matrix coefficients modulo $\mathfrak{a}$. We say that $\Lambda < \Gamma$ is a \emph{congruence subgroup} if $\ker r_\mathfrak{a} < \Lambda$ for some ideal $\mathfrak{a}$. In particular, if $\Lambda$ is a congruence subgroup of $\Gamma$, then there exists an ideal $\mathfrak{a} < \mathcal{O}_K$ and a subgroup $G < \PSL(2,\mathcal{O}_K/\mathfrak{a})$ such that $\Lambda = R_\mathfrak{a}^{-1}(G)$. 

In what follows, we assume that $\Lambda$ is a finite index congruence subgroup of $\Gamma$. By the Strong Approximation Theorem (see \cite{Weis}), there is a cofinite set of unramified primes $\mathcal{P}_\Gamma$ of $K$ such that the homomorphism $R_{\mathfrak{P}^j}$ is surjective for all $\mathfrak{P} \in \mathcal{P}_\Gamma$ and $j \in \mathbf{N}$. Additionally if $\mathfrak{a}$ is an ideal of $\mathcal{O}_K$ with decomposition $\mathfrak{a} = \mathfrak{P}_1^{\alpha_1} \dots \mathfrak{P}_r^{\alpha_r}$, where $\mathfrak{P}_i$ are primes in $K$ and $\alpha_i\in \mathbf{N}$, then the Chinese Remainder Theorem yields $\PSL(2,\mathcal{O}_K/\mathfrak{a}) \cong \prod_{i=1}^r \PSL(2,\mathcal{O}_K/\mathfrak{P}_i^{\alpha_i})$, and $R_\mathfrak{a} = R_{\mathfrak{P}_1^{\alpha_1}} \times \dots \times R_{\mathfrak{P}_r^{\alpha_r}}$. As mentioned in Proposition \ref{proposition:totallygeodesicinvariants}, each Fuchsian subgroup $\Delta$ of $\Gamma$ is defined over a totally real subfield $k < K$ with $[K:k] = 2$. Again by the Strong Approximation Theorem, there is a cofinite set of primes of $k$ such that $R_{\mathfrak{p}^j}\colon \Delta \to \PSL(2,\mathcal{O}_k/\mathfrak{p}^j)$ is surjective for all $j \in \mathbf{N}$. Viewing $\Delta < \Gamma$, this provides us with a cofinite set of primes $\mathfrak{p}$ of $k$ such that $R_\mathfrak{P}(\Delta) \cong \PSL(2,\mathcal{O}_k/\mathfrak{p})$ for all $\mathfrak{P} \in \mathcal{P}_\Gamma$, where $\mathfrak{p}$ is the prime lying under $\mathfrak{P}$ in $k$. We denote by $\mathcal{P}_\Delta$, the subset of this cofinite set such that $\mathfrak{P}\in\mathcal{P}_\Gamma$ whenever $\mathfrak{p}\in\mathcal{P}_\Delta$ where $\mathfrak{P}$ is any prime lying over $\mathfrak{p}$. It is clear that $\mathcal{P}_\Delta$ is also a cofinite subset of the set of primes of $k$.

Returning to our study of Fuchsian subgroups, let $\Delta < \Gamma$ be a fixed Fuchsian subgroup. For primes $\mathfrak{p}$ in $\mathcal{P}_\Delta$, the primary decomposition of the ideal $\mathfrak{p}\mathcal{O}_K$ is either $\mathfrak{p}\mathcal{O}_K = \mathfrak{P}$ or $\mathfrak{P}_1\mathfrak{P}_2$. In the first case, we say that $\mathfrak{p}$ is inert and in the second case we say that $\mathfrak{p}$ is split. We additionally have that
\begin{equation}\label{equation:residuefieldextension}
 [\mathcal{O}_K/\mathfrak{P}:\mathcal{O}_k/\mathfrak{p}] = \begin{cases} 1, & \mathfrak{p} \textrm{ is split}, \\ 2, & \mathfrak{p} \textrm{ is inert.} \end{cases}
\end{equation}

\begin{lemma}\label{CC-Sec:L2}
There is a set of primes $\mathcal{Q}_\Delta$ which is cofinite in $\mathcal{P}_\Delta$ such that for each $\mathfrak{p}\in\mathcal{Q}_\Delta$ there exists $\gamma \in \Gamma$ with $R_{\mathfrak{P}^j}(\gamma^{-1}\Delta \gamma) = \PSL(2,\mathcal{O}_k/\mathfrak{p}^j)$.
\end{lemma}

\begin{proof}
If $\mathfrak{p}$ is a split prime, then by virtue of \eqref{equation:residuefieldextension} we have $R_{\mathfrak{P}^j}(\Delta) = \PSL(2,\mathcal{O}_k/\mathfrak{p}^j)=\PSL_2(\mathcal{O}_K/\mathfrak{P}^j)$, and the result is immediate. We therefore assume that $\mathfrak{p}$ is inert and $\mathfrak{P}$ is the unique prime of $K$ lying over $\mathfrak{p}$. Using the notation $\mathbf{F}_q=\mathcal{O}_k/\mathfrak{p}$ and $\mathbf{F}_{q^2}=\mathcal{O}_K/\mathfrak{P}$, if $q$ is odd (i.e. $\mathfrak{p}$ is not a dyadic prime) results of Dickson \cite{Dickson} give that there are two conjugacy classes of subgroups isomorphic to $\PSL(2,\mathbf{F}_q)$ inside of $\PSL(2,\mathbf{F}_{q^2})$. Writing $\mathbf{F}_{q^2}=\mathbf{F}_q[\theta]$, then the first conjugacy class is induced from the field embedding $\mathbf{F}_q<\mathbf{F}_{q^2}$ and the second conjugacy class is obtained by conjugating the standard embedding by the diagonal matrix $\mathrm{diag}(1,\theta)$. We claim that for all but finitely many inert primes $\mathfrak{p}$, the image of $R_\mathfrak{P}(\Delta)$ lies in the former conjugacy class. To this end let $A/K$, $B/k$ be the quaternion algebras from which $\Gamma$, $\Delta$ arise (respectively), where $B\otimes_kK\cong A$. Let $\mathcal{O}$, $\mathcal{O}'$ be maximal orders of $A$, $B$ such that $\mathcal{O}'\otimes_{\mathcal{O}_k}\mathcal{O}_K\cong\mathcal{O}$ and such that $\Gamma$, $\Delta$ are commensurable with $P(\pi(\mathcal{O}^1))$, $P(\pi(\mathcal{O}'^1))$, respectively. Then define $\mathcal{I}_k$ to be the subset of inert, unramified primes $\mathfrak{p}$ in $k$ such that the following hold:
\begin{enumerate}
\item $\mathfrak{p}$ is non-dyadic,
\item $B$ is not ramified at $\mathfrak{p}$,
\item The closure of $\Delta$ in $P((B\otimes_{k}k_{\mathfrak{p}})^1)$ is isomorphic to $\PSL(2,\mathcal{O}_{k_\mathfrak{p}})$,
\item The closure of $\Gamma$ in $P((A\otimes_{K}K_\mathfrak{P})^1)$ is isomorphic to $\PSL(2,\mathcal{O}_{K_\mathfrak{P}})$, where $\mathfrak{P}$ lies over $\mathfrak{p}$,
\item $P((\mathcal{O}'\otimes_{k}\mathcal{O}_{k_\mathfrak{p}})^1)= \PSL(2,\mathcal{O}_{k_{\mathfrak{p}}})$.
\end{enumerate}
It is immediate that the set of inert primes satisfying (1)-(4) is cofinite in the set of inert primes of $k$. For $\mathfrak{p}$ also satisfying condition (5), note that just as in the finite field case there can be two conjugacy classes of $\PSL(2,\mathcal{O}_{k_{\mathfrak{p}}})$ in $\PSL(2,k_\mathfrak{p})$. However, we must have that $P((\mathcal{O}'\otimes_{k}\mathcal{O}_{k_\mathfrak{p}})^1)= \PSL(2,\mathcal{O}_{k_{\mathfrak{p}}})$ for almost all $\mathfrak{p}$. Indeed, if this were not the case then $\Delta$ and $P(\pi(\mathcal{O}^1))$ would not have finite index intersection. Therefore $\mathcal{I}_k$ is a cofinite subset of the set of inert primes of $k$. Moreover, for $\mathfrak{p}\in\mathcal{I}_k$ we have the inclusion
\[ \Delta_{\mathfrak{p}}=P((\mathcal{O}\otimes_{\mathcal{O}_k}\mathcal{O}_{k_\mathfrak{p}})^1)= \PSL(2,\mathcal{O}_{k_{\mathfrak{p}}})<\PSL(2,\mathcal{O}_{K_{\mathfrak{P}}})=P((\mathcal{O}\otimes_{\mathcal{O}_K}\mathcal{O}_{K_\mathfrak{P}})^1)=\Gamma_{\mathfrak{P}} . \]
The form of the Strong Approximation Theorem from \cite[Cor 16.4.3]{LubotzkySegal} implies that for any prime $\mathfrak{P}$ lying over a prime $\mathfrak{p}\in\mathcal{I}_k$, $R_{\mathfrak{P}}(\Delta)=R_{\mathfrak{P}}(\Delta_\mathfrak{p})=\PSL(2,\mathbb{F}_q)$. Therefore $\mathcal{Q}_\Delta$ certainly contains the union of the split primes and $\mathcal{I}_k$, and thus is a cofinite subset of $\mathcal{P}_\Delta$ with the requisite property.
\end{proof}

\begin{lemma}\label{CC-Sec:L3}
If $\mathfrak{p}\in\mathcal{P}_\Delta$ and $\mathfrak{P}$ is a prime of $K$ lying over $\mathfrak{p}$, then $R_{\mathfrak{P}^j}(\Delta)$ is a maximal subgroup of $R_{\mathfrak{P}^j}(\Gamma)$.
\end{lemma}

\begin{proof}
This follows from Dickson \cite{Dickson}.
\end{proof}

As all Fuchsian subgroups of $\Gamma$ are defined over the same field $k$, we may define
\[ \mathcal{Q}_\mathcal{A}=\bigcap_{\substack{\Delta<\Gamma\\ \coarea(\Delta)\in\mathcal{A}}}\mathcal{P}_\Delta. \]
Note that $\mathcal{Q}_\mathcal{A}$ is cofinite in the set of primes of $k$ by virtue of the fact that the number of $\Gamma$--conjugacy classes of Fuchsian subgroups with coarea in $\mathcal{A}$ is necessarily finite. In the rest of this section, we will always assume that $\mathfrak{p}\in\mathcal{Q}_\mathcal{A}$ and that $\mathfrak{P}$ always lies over some such $\mathfrak{p}$.

Given an ideal $\mathfrak{a}$ in $\mathcal{O}_K$, write its primary decomposition as $\mathfrak{a} = \mathfrak{P}_1^{\alpha_1}\dots \mathfrak{P}_r^{\alpha_r}$. Then for each $i \in \set{1,\dots,s}$, there exists $\Delta_i $ such that $R_\mathfrak{a}(\Delta_i) \cong \PSL(2,\mathcal{O}_k/\mathcal{O}_k \cap \mathfrak{a})$ and $R_\mathfrak{a}(\Gamma) = \PSL(2,\mathcal{O}_K/\mathfrak{a})$. By the Chinese Remainder Theorem and Lemma \ref{CC-Sec:L2}, there exists $\gamma_i \in \Gamma$ such that $R_\mathfrak{a}(\gamma_i^{-1}\Delta_i\gamma_i) = \PSL(2,\mathcal{O}_k/\mathcal{O}_k \cap \mathfrak{a})$. Setting $\Lambda_\mathfrak{a} = R_\mathfrak{a}^{-1}(\PSL(2,\mathcal{O}_k/\mathcal{O}_k \cap \mathfrak{a}))$, we see that $\Lambda_\mathfrak{a}$ has Fuchsian subgroups $\gamma_i^{-1}\Delta_i\gamma_i$ with co-area $A_1,\dots,A_s$.

We will say that a congruence subgroup $\Lambda$ of $\Gamma$ \emph{has level in $\mathcal{Q}_\mathcal{A}$} if $\Lambda = R_\mathfrak{a}^{-1}(G)$ for some ideal $\mathfrak{a} < \mathcal{O}_K$ and some subgroup $G < R_\mathfrak{a}(\Gamma)$ where $\mathfrak{a} = \mathfrak{P}_1^{\alpha_1}\dots\mathfrak{P}_r^{\alpha_r}$, with each prime $\mathfrak{P}_i$ lying over a prime in $\mathcal{Q}_\mathcal{A}$. Moreover for each $\mathfrak{P}_j$, it is clear that we may assume that $G$ does not surject $R_{\mathfrak{P}_j^{\alpha_j}}(\Gamma)$ as otherwise the prime $\mathfrak{P}_j$ would be irrelevant to our decomposition of $\mathfrak{a}$. If $\Lambda$ is a congruence subgroup of $\Gamma$ with level in $\mathcal{Q}_\mathcal{A}$ that contains Fuchsian subgroups $\Delta_1',\dots,\Delta_s'$ with $\coarea(\Delta_i') = A_i$, we assert that $\Lambda$ is conjugate in $\Gamma$ to $\Lambda_\mathfrak{a}$ from some ideal $\mathfrak{a}$. Indeed, as $R_\mathfrak{a}(\Lambda) = G < \PSL(2,\mathcal{O}_K/\mathfrak{a})$ for some $G$ and some ideal $\mathfrak{a} < \mathcal{O}_K$, we see that $R_\mathfrak{a}(\Delta_i') < G$ for all $i=1,\dots,s$. By Lemma \ref{CC-Sec:L3}, the subgroups $R_\mathfrak{a}(\Delta_i')$ are maximal and so $G = R_\mathfrak{a}(\Delta_1') = \dots = R_\mathfrak{a}(\Delta_s')$. Finally, by Lemma \ref{CC-Sec:L2}, $G$ is conjugate in $\PSL(2,\mathcal{O}_K/\mathfrak{a})$ to $\PSL(2,\mathcal{O}_k/\mathcal{O}_k \cap \mathfrak{a})$. Taking $g \in \PSL(2,\mathcal{O}_K/\mathfrak{a})$ to be such a conjugating element and taking $\gamma = R_\mathfrak{a}^{-1}(g)$, we see that $\gamma^{-1} \Lambda \gamma = \Lambda_\mathfrak{a}$. 

When $\mathfrak{p} \in \mathcal{Q}_\mathcal{A}$ and $\mathfrak{p}$ is split, then for any $\Delta$ with coarea in $\mathcal{A}$ we have $R_{\mathfrak{P}^j}(\Delta) = R_{\mathfrak{P}^j}(\Gamma)$. In particular, when $\mathfrak{p} \in \mathcal{Q}_\mathcal{A}$ is split we have the equality $\Gamma = \Lambda_{\mathfrak{p}^j}$. Combined with the above, we then have the following consequence.

\begin{thm}\label{CC-Sec:T1}
Let $\Lambda < \Gamma$ be a congruence cover with level in $\mathcal{Q}_\mathcal{A}$ such that there exist Fuchsian subgroups $\Delta_i' < \Lambda$ for $i=1,\dots,s$ with $\coarea(\Delta_i') = A_i$. Then there exists an ideal $\mathfrak{a}$ in $\mathcal{O}_K$ with primary decomposition $\mathfrak{a} = \mathfrak{P}_1^{\alpha_1} \dots \mathfrak{P}_r^{\alpha_r}$, and a subgroup $G < \PSL(2,\mathcal{O}_K/\mathfrak{a})$ with $G \cong \PSL(2,\mathcal{O}_k/\mathcal{O}_k \cap \mathfrak{a})$. Moreover, the primes $\mathfrak{p}_i$ of $k$ lying under $\mathfrak{P}_i$ are inert and each $\mathfrak{p}_i \in \mathcal{Q}_\mathcal{A}$.
\end{thm}

As a result of Theorem \ref{CC-Sec:T1}, we obtain a complete classification of the congruence subgroups $\Lambda$ of $\Gamma$ with level in $\mathcal{Q}_\mathcal{A}$ that contain Fuchsian subgroups $\Delta_1',\dots,\Delta_s'$ with $\coarea(\Delta_i') = A_i$. We define $\mathcal{Q}_{inert}$ to be the subset of $\mathcal{Q}_\mathcal{A}$ of primes $\mathfrak{p}$ that are inert and we define $\mathcal{Q}_{K,inert}$ to be the set of primes in $K$ lying over primes $\mathfrak{p}$ in $\mathcal{Q}_{inert}$. From the discussion in this section, we now prove Theorem \ref{CC-Sec:T2}.

\begin{proof}[Proof of Theorem \ref{CC-Sec:T2}]
The set of congruence subgroups in Theorem \ref{CC-Sec:T1} is in bijection with the set of ideals
\[ \mathrm{I}_\mathcal{A} = \set{\mathfrak{P}_1^{\alpha_1}\dots\mathfrak{P}_r^{\alpha_r}~:~\mathfrak{P}_i \in \mathcal{Q}_{K,inert},~\alpha_i \in \mathbf{N}}. \]
For each prime $\mathfrak{P} \in \mathcal{Q}_{K,inert}$ with underlying prime $\mathfrak{p}$, we let $q = \abs{\mathcal{O}_k/\mathfrak{p}}$ and note that $\abs{\mathcal{O}_K/\mathfrak{P}} = q^2$. We know that
\[ \abs{\PSL(2,\mathcal{O}_k/\mathfrak{p}^j)} = \frac{q^{3(j-1)}q(q-1)(q+1)}{2} \approx q^{3j}, \quad \abs{\PSL(2,\mathcal{O}_K/\mathfrak{P}^j)} = \frac{q^{6(j-1)}q^2(q^2-1)(q^2+1)}{2} \approx q^{6j}. \] 
Hence, we see that $[\Gamma:\Lambda_{\mathfrak{p}^j}] \approx \N(\mathfrak{p})^{3j}=\N(\mathfrak{p}^j)^{3}$, and so by the Chinese remainder theorem, we also have $[\Gamma:\Lambda_\mathfrak{a}] \approx \N(\mathfrak{a})^{3}$ when $\mathfrak{a} \in \mathrm{I}_\mathcal{A}$. In particular, for each $X \in \mathbf{N}$ and 
\[ \mathrm{L}_{\Gamma,\mathcal{A}}(X) = \set{\Lambda \in \mathcal{L}(\Gamma,\mathcal{A})~:~[\Gamma:\Lambda] \leq X,~\Lambda \textrm{ is congruence}}, \]
we have that $\mathrm{L}_{\Gamma,\mathcal{A}}(X) \geq \abs{\set{\mathfrak{a} \in \mathrm{I}_\mathcal{A}~:~\N(\mathfrak{a}) \leq X^{1/3}}}$. By \cite{Coleman}, we have that 
\[ \left|\left\{\mathfrak{a}\in\mathrm{I}_\mathcal{A}~:~\N(\mathfrak{a})\le X\right\}\right|\sim \frac{cX^{1/2}}{(\log X)^{1/2}} \]  
as $X$ goes to infinity, where $c$ is some positive constant which depends only on the field $k$. Hence, for $X \gg 0$, we have $\mathrm{L}_{\Gamma,\mathcal{A}}(X) \gg \frac{X^{1/6}}{(\log X)^{1/2}}$, where the implicit constant depends only on $k$. Now assume that $V\ge \vol(M)$. Letting $\mathrm{Cov}(M)$ be the set of all manifolds which finitely cover $M$, this gives that
\[ H(V)=\left|\left\{M'\in\mathrm{Cov}(M)~:~\mathcal{A}\subset\TGA(M'),~\vol(M')\le V\right\}\right|\ge \mathrm{L}_{\Gamma,\mathcal{A}}\left(\left\lfloor\frac{V}{\vol(M)}\right\rfloor\right) \gg \frac{\!(V/\vol(M)) ^{1/6}}{\log(V/\vol(M))^{1/2}}, \]
where the implicit constant depends only on $k$. As the field $k$ depends solely on $K$, we can therefore transfer the dependence of the implicit constant to $K$. This completes the proof.
\end{proof}


\section{Final Remarks}

The geometric genus spectrum is a constituent of a finer invariant associated to a closed, hyperbolic 3--orbifold $M$. In \cite{MR2}, the second author and Reid studied the $\mathrm{Isom}(\mathbf{H}^3)$--conjugacy classes of injective homomorphisms of surface groups to $\pi_1(M)$; we will call this the \emph{total surface spectrum}. Again by work of Thurston \cite{Thurston-Notes}, for each finite type surface $\Sigma$ there are only finitely many injective homomorphisms $\pi_1(\Sigma) \to \pi_1(M)$ up to $\mathrm{Isom}(\mathbf{H}^3)$--conjugacy. Hence, the total surface spectrum of $M$ has finite multiplicities. 

These homomorphisms satisfy a dichotomy given by whether the image of the surface group is geometrically finite or infinite. Totally geodesic surfaces comprise a (proper) subset of the geometrically finite homomorphisms. Though there need not be a totally geodesic surface subgroup of $\pi_1(M)$, the geometrically finite part of this invariant is well stocked (e.g.~infinite) for every closed hyperbolic 3--orbifold by groundbreaking work of Kahn--Markovic \cite{KM}. By the work of Agol \cite{Agol}, Bonahon \cite{Bon}, Calegari--Gabai \cite{CG}, and Thurston \cite{Thurston-Notes}, the geometrically infinite surface group homomorphisms are precisely those homomorphisms that have virtually fibered image. By separate, groundbreaking work of Agol \cite{Agol2}, the set of virtually fibered surface subgroups of $\pi_1(M)$ is also infinite.

If $M_1,M_2$ are a pair of complete, finite volume, hyperbolic $3$--orbifolds with the same total surface spectrum, then $M_1,M_2$ have both the same set of geometrically finite and geometrically infinite surface subgroups. We will call the submultisets of the total surface spectrum associated to either the geometrically finite or geometrically infinite surfaces as the finite and infinite surface spectra.  At present, there are no known examples of non-isometric pairs of complete, finite volume, hyperbolic $3$--orbifolds $M_1,M_2$ with the same set of surface subgroups. Such a pair would also provide examples of pairs with the same geometrically finite surface subgroups and the same virtually fibered surface subgroups. We note that presently there are no known examples with either the same set of geometrically finite or virtually fibered surface subgroups as well. The latter case is connected to the Short Geodesic Conjecture. In \cite{MR2}, it was shown that if the Short Geodesic Conjecture (see \cite{LMPT} for more on this conjecture) for arithmetic hyperbolic 3--orbifolds holds, then the set of genera of those surfaces that arise as virtual fibers of an arithmetic hyperbolic 3--manifold $M$ determines the commensurability class of $M$. In particular, non-commensurable pairs of arithmetic hyperbolic $3$--manifolds with the same set of virtually fibered surfaces would imply that the Short Geodesic Conjecture is false. Though we believe that the Short Geodesic Conjecture holds, this connection provides additional, broad motivation for our work here. Additionally, it partially motivates the following questions.

\textbf{Question 1.} Let $M_1,M_2$ be a pair of complete, finite volume, hyperbolic $3$--orbifolds.
\begin{itemize}
\item[(1)]
If $M_1,M_2$ have the same total surface spectrum, are $M_1,M_2$ isometric?
\item[(2)]
If $M_1,M_2$ have the same finite surface spectrum, are $M_1,M_2$ isometric?
\item[(3)]
If $M_1,M_2$ have the same infinite surface spectrum, are $M_1,M_2$ isometric?
\end{itemize}


\end{document}